\documentclass{amsart}

\usepackage{amssymb}
\usepackage{amscd}
\usepackage{latexsym}
\usepackage{amsfonts}
\usepackage{amsmath}
\usepackage[utf8]{inputenc}

\usepackage[all,line]{xy} 

\newtheorem{proposition}{Proposition}
\newtheorem{lemma}{Lemma}
\newtheorem{theorem}{Theorem}
\newtheorem{corollary}{Corollary}
\newtheorem{definition}{Definition}

\newcommand{\tr}{\mathop{\fam0 Tr}\nolimits}

\newcommand{\Diff}{\mathop{\fam0 Diff}\nolimits}
\newcommand{\Hom}{\mathop{\fam0 Hom}\nolimits}
\newcommand{\End}{\mathop{\fam0 End}\nolimits}

\newcommand{\id}{\mathop{\fam0 Id}\nolimits}
\newcommand{\Lie}[1]{\mbox{\sf #1}}

\newcommand{\Aut}{\mathop{\fam0 Aut}\nolimits}

\newcommand{\Id}{\mathop{\fam0 Id}\nolimits}

\newcommand{\R}[1]{{\mathbb R}^{#1}}
\newcommand{\bC}{{\mathbb C}}
\newcommand{\bR}{{\mathbb R}}

\newcommand{\C}{C}

\newcommand{\Z}{{\mathbb Z}}
\newcommand{\N}{{\mathbb N}}

\newcommand{\ra}{\mathop{\fam0 \rightarrow}\nolimits}

\newcommand{\cH}{ {\mathcal H}}

\newcommand{\T}{ {\mathcal T}}

\renewcommand{\P}{ {\mathbb P}}

\newcommand{\delbar}{\bar{\partial}}
\newcommand{\del}{\partial}
\newcommand{\set}[1]{\left\{ #1\right\}} 
\newcommand{\bH}{\mathbb{H}}

\newcommand{\cv}[1]{\nabla_{#1}} 
\newcommand{\lp}[1]{\Delta_{#1}} 
\newcommand{\HC}{\boldsymbol{\nabla}} 
\DeclareMathOperator{\ima}{Im} 
\newcommand{\iprod}[1]{\left\langle #1 \right\rangle} 
\newcommand{\pk}[0]{\pi^{(k)}} 

\begin{document}

\title[Hitchin Connection for a large class of complex structures]{A Hitchin Connection for a large class\\ of families of Kähler Structures}

\author{J{\o}rgen Ellegaard Andersen \& Kenneth Rasmussen}
\address{Department of Mathematics\\
        University of Aarhus\\
        DK-8000, Denmark}
\email{andersen@imf.au.dk}

\thanks{Supported in part by the center of excellence grant "Center for quantum geometry of Moduli Spaces" (DNRF95) from the Danish National Research Foundation. }

\begin{abstract}
In this paper we construct a Hitchin connection in a setting, which significantly generalizes the setting covered by the first author in \cite{A5}, which in turn was a generalisation of the moduli space case covered by Hitchin in his original work on the Hitchin connection \cite{H}. In fact, our construction provides a Hitchin connection, which is a partial connection on the space of all compatible complex structures on an arbitrary, but fixed prequantizable symplectic manifold, which satisfies a certain Fano type condition. The subspace of the tangent space to the space of compatible complex structures on which the constructed Hitchin connection is defined, is in fact of finite co-dimension, if the symplectic manifold is compact. In a number of examples, including flat symplectic space, symplectic tori and moduli spaces of flat connections for a compact Lie group, we prove that our Hitchin connection is defined in a neighbourhood of the natural families of complex structures compatible with the given symplectic form, which these spaces admits.
\end{abstract}

\maketitle

\begin{center}
\thanks{Dedicated to Nigel Hitchin at the conference {\em Hitchin70},\\ celebrating his 70'th Birthday.}
\end{center}

\section{Introduction}

In this paper we construct a Hitchin connection in a setting, which generalises earlier work done by the first author in \cite{A5}, where the Hitchin connection is constructed under the rather restrictive assumption that the family of complex structures has the so called rigid property, which was also the case for the moduli space case, in which Hitchin constructed his connection first \cite{H}. This means that the corresponding deformations of the metric is by the real part of a global holomorphic symmetric tensor, as we recall in details below. In particular if a given K\"{a}hler manifold has no global holomorphic symmetric tensors beside zero, then the approach of \cite{A5} does not apply, hence the wording rigid for such families - they constitute rather thin slices in the infinite dimensional space of all complex structures. In this paper we relax this condition considerably.

In order to describe our generalisation, let us briefly introduce the setting. 
We let $(M,\omega)$ be a symplectic manifold. We let $\mathcal{T}$ be a complex manifold parametrizing a holomorphic family $J\colon \mathcal{T} \to C^\infty(M, \End(TM))$ of complex structures, which are all Kähler with respect to $\omega$. We will write $M_\sigma$, when we refer to the complex manifold $(M,J_\sigma)$, where $\sigma$ is any point in $\mathcal{T}$.

We will consider the variation of the family $J$ along a real vector field $V$ on $\mathcal{T}$, which we denote $V[J]$. We consider the splitting of $V = V' + V''$ into types on $\mathcal{T}$ and we consider the symmetric bi-vector field $G(V')=V'[J]\cdot\tilde{\omega}$, where $\tilde{\omega}$ is the bivector field, inverse to $\omega$. We think of $G$ as a one form on $\mathcal{T}$ with coefficients in bi-vector fields and as such, we write $G(V) = G(V')$. Observe that if $g$ is the corresponding family of K\"{a}hler metrics parametriced by $\mathcal{T}$, then we have that
$$ V[g] = G(V) + \overline{G(V)}.$$

The assumption that the {\em family} $J$ is rigid, says that $G(V)$ defines a holomorphic section $G(V)_\sigma\in H^0(M_\sigma,S^2(T'M_\sigma))$ at all points $\sigma \in \mathcal{T}$. This is of course a very restrictive condition, however, it is satisfied in the setting, in which Hitchin initially introduced his connection, which was the case of Teichm\"{u}ller space parametrising K\"{a}hler structures for the Seshadri-Atiyah-Bott-Goldman symplectic form \cite{AB,NS1,NS2} on the moduli spaces of flat $SU(n)$ connections on a genus $g$ surface \cite{H}. See also the work of Axelrod, Della Pietra and Witten for a physical derivation of this connection using Chern-Simons theory \cite{ADW} and \cite{A5} for a verification that the two connections agree.

In this paper, we weaken the rigid criterion by adding the possibility of varying the bi-vector field $G(V)$ by adding  a term of the form $\delbar\beta(V)\cdot \tilde \omega$ for an arbitrary vector field $\beta(V)_\sigma\in C^\infty(M_\sigma, T'M_\sigma)$. 

\begin{definition}
We call the family {\em weakly restricted} if there exist a one form $\beta$ on $\mathcal{T}$ with values in $C^\infty(M_\sigma, T'M_\sigma)$ at each point $\sigma\in \mathcal{T}$, such that for all vector fields $V$ along $\mathcal{T}$ and all $\sigma\in\mathcal{T}$, there exist $G_\beta(V)_\sigma \in H^0(M, S^2(T'M_\sigma))$ such that 
\begin{equation}\label{Gb}
G_{\beta}(V)_\sigma\cdot \omega = V'[J]_\sigma + \delbar \beta(V)_\sigma.
\end{equation}
\end{definition}

The main result in this article is the construction of a Hitchin connection, when we assume the family to be weakly restricted, on top of a couple of further minor topological assumptions.

It is of course interesting to investigate, when we can solve the weakly restricted criterion. We let $\mathcal{C}_\omega$  be the space of all complex structures on $M$ compatible with the symplectic form $\omega$ and let $J\in \mathcal{C}_\omega(M)$. Then we have that
\begin{equation*}
  T_J\mathcal{C}_\omega = \ker( \delbar_J \colon \Omega^{0,1}(M,T'M_J)_\omega \to \Omega^{0,2}(M,T'M_J)),
\end{equation*} 
where 
\begin{equation}\label{sym01}
\Omega^{0,1}(M,T'M_J)_\omega = \{ V_J\in \Omega^{0,1}(M,T'M_J) \mid \omega(V_J \cdot, J \cdot) = \omega(\cdot, J V_J \cdot)\},
\end{equation}
which is the same as stating that $V_J$ is symmetric with respect to the K\"{a}hler metric $g_J$ associated to $\omega$ and $J$. Thus,  we see that given $V_J\in T_J\mathcal{C}_\omega$, we can solve the weakly restricted condition, e.g. find $\beta(V)$, whenever we have
\begin{align*}
  V_J \in H^0(M_J, S^2(T'M_J))\cdot\omega +
  &\ima(\delbar_J \colon C^\infty(M,T'M_J)_\omega \to \Omega^{0,1}(M,T'M_J)_\omega).
\end{align*}
where 
$$C^\infty(M,T'M_J)_\omega = \{ X\in C^\infty(M,T'M_J) \mid \bar\partial (i_X\omega) =0 \}.
$$	
Thus if the map
$$ \cdot \omega : H^0(M_J, S^2(T'M_J)) \ra  H^1(M_J, T'M_J)_\omega$$
is surjective, this is always possible. Here $H^1(M_J, T'M_J)_\omega$ is defined in analogy with (\ref{sym01}), namely to be the symmetric part of this cohomology.

A particular simple case, where we can always solve (\ref{Gb}) is of course if
\begin{align*}
  H^1(M_J, T'M_J)_\omega = 0.
\end{align*}
In general we can solve the equation (\ref{Gb}) if the cohomology class of $V'[J]_\sigma$ is contained in the image of $\cdot\omega$. Thus our construction will only provide a partial connection on the space of all complex structures compatible with the symplectic form on $\omega$. If $M$ is compact, we see that this partial connection is defined on a subspace of finite co-dimension of the tangent space to the space of all complex structures compatible with $\omega$.

Let us now briefly recall the setup in geometric quantization. Let $(M, \omega)$ be a symplectic manifold and assume that $(M, \omega)$ admits a prequantum line bundle $(\mathcal{L},\nabla, \langle\cdot,\cdot\rangle)$ \cite{W}. Let $\mathcal{T}$ be a complex manifold parametrizing a holomorphic family of complex structures $J$ making $(M, \omega, J_\sigma)$ Kähler for each $\sigma \in \mathcal{T}$. Now for each $\sigma\in\mathcal{T}$ we consider the quantum space at level $k\in \N$, which is the subspace $H^{(k)}_\sigma$ of the prequantum space $\mathcal{H}^{(k)} = C^\infty(M, \mathcal{L}^k)$ consisting of holomorphic sections
$$ H^{(k)}_\sigma = H^0(M_\sigma, \mathcal{L}^k) \subset \mathcal{H}^{(k)}.$$
We will assume that these quantum spaces form a smooth subbundle $H^{(k)}$ of the trivial bundle
\begin{equation*}
  \hat{\mathcal{H}}^{(k)}=\mathcal{T}\times \mathcal{H}^{(k)}.
\end{equation*}

Now we let $\cv{}^T$ denote the trivial connection on $\hat{\mathcal{H}}^{(k)}$, and then we consider a connection of the form 
\begin{equation}
  \label{eq:18}
  \HC_V=\cv{V}^T + u(V),
\end{equation}
where $u\in \Omega^{1}(\mathcal{T},\mathcal{D}(M,\mathcal{L}^k))$ is a one-form on $\mathcal{T}$ with values in the space of differential operators on sections of $\mathcal{L}^k$. Our goal is to construct a $u$, such that $\HC$ preserves the quantum spaces $H^{(k)}_\sigma$ inside each fiber of $\hat{\cH}^{(k)}$.
\begin{definition}[Hitchin connection]
  \label{def:2}
  A Hitchin connection in the bundle $\hat{\cH}^{(k)}$ is a connection of the form \eqref{eq:18}, that preserves the subspaces $H^{(k)}_\sigma$ inside each fiber of $\hat{\cH}^{(k)}$. 
\end{definition}

We prove the following theorem in this paper.

\begin{theorem}[Hitchin connection for weakly restricted families]
  \label{thm:2}
  Let $(M, \omega)$ be a symplectic manifold with a  prequamtum line bundle $\mathcal{L}$. Assume that $M$ has first Chern class of the form $c_1(M,\omega)= n \left[\frac{\omega}{2\pi} \right]$ for some integer $n \in \Z$ and such that $b_1(M) = 0$. Furthermore, let $J\colon \mathcal{T} \to C^\infty (M, \End(TM))$ be a weakly restricted, holomorphic family of Kähler structures on M, parametrized by a complex manifold $\mathcal{T}$, and assume that the family admits a family of Ricci potentials $F$.
  Then there exists a Hitchin connection $\HC$ in the bundle $\hat{H}^{(k)}$ over $\mathcal{T}$, given by the expression
  $$   \HC_V= \cv{V}^T + u(V), $$ 
  where
    \begin{align}
    \label{eq:6}
 u(V)&=\frac{1}{2(2k+n)}(\lp{G_{\beta}(V)} + 2 \cv{G_{\beta}(V)\cdot dF}-i(2k+n) \cv{\beta(V)} \nonumber\\
      &+4k V'[F] -2ik dF \cdot \beta(V) -ik\delta(\beta(V)) +2k(k+n)\phi(V)+ik\psi(V)), \nonumber
  \end{align}
  and $\phi(V),\psi(V) \in C^\infty(M)$ are smooth functions, satisfying
$$\delbar \phi(V)=\omega\cdot \beta(V) \quad \text{and} \quad \delbar\psi(V)=\Omega(V),$$
where $\Omega(V) \in \Omega^1(M)$ is given by
  \begin{equation*}
    \Omega(V)
    =-\delta(G_{\beta}(V))\cdot \omega +\delta(\delbar\beta(V)) - 2dF \cdot G_{\beta}(V)\cdot \omega + 2\delbar \beta(V) \cdot dF+ 4i\delbar V'[F].
  \end{equation*}
Furthermore, if non of the complex structures admit non-constant holomorphic functions on $M$, which is true for instance if $M$ is compact, we get that $\psi(V)=0$.
 \end{theorem}

We also remark that in the case where $M$ is compact, Hodge theory will provide us with a family of Ricci potentials and of course there will in that case only be constant holomorphic functions globally on $M$, so these two assumptions can be ignored in the compact case, reducing the assumptions to only two cohomological restrictions. We of course expect that the Fano type condition
$$c_1(M,\omega)= n \left[\frac{\omega}{2\pi} \right]$$
can be removed by doing metaplectic correction as considered in \cite{AGL}.

The condition $b_1(M)=0$ is only used to ensure that the closed 1-form $\Omega(V)$ is exact, such that $\psi(V)$ exists. In other words, if we already have a $\bar\partial$-primitiv for $\Omega(V)$ for all $V$, this assumption can also be ignored. See details in the proof of proposition \ref{pro:1}.

Observe that when $\beta(V)=0$, the family is rigid and this new Hitchin connection restricts to the Hitchin connection in \cite{A5}.

We stress that we do not need that $\mathcal{T}$ is a complex manifold, in fact we have a complete analog of Theorem \ref{thm:2} in this case. Please see Theorem \ref{thm:3} in section \ref{sec:hitch-conn-smooth}.

In the case of rigid families, it was proved that the Hitchin connection in this setting is projectively flat by the first author and Gammelgaard in \cite{AG1} generalizing the projective flatness proofs of Hitchin in \cite{H} and Axelrod, Della Pietra and Witten in \cite{ADW} in the original moduli space setting (see also \cite{vGdJ,R1}). The projective flatness is of course very importance for its relation to quantum Chern-Simons theory, in particular in relation to the projective representations of the mapping class groups which rises. These actually are the same as the Witten-Reshetikhin-Turaev TQFT (WRT-TQFT) \cite{RT1,RT2,T,B1,BHMV1,BHMV2} representations as it follows by the combination of two results. One by Laszlo \cite{La1}, which identifies the Hitchin connection projectively with the TUY-connection, constructed in conformal field theory by Tsuchiya, Ueno and Yamada \cite{TUY}. The second by the first author and Ueno, which identifies the TQFT's comes from conformal fields theory for the affine Lie-algebra of $\Lie{sl}(m,{\mathbb C})$ and then the WRT-TQFT constructed from $U_q(\Lie{sl}(m,{\mathbb C}))$, for $q= \exp(2\pi i /(k+n))$ \cite{AU1,AU2,AU3,AU4} as proposed in \cite{W1}. This has been exploited to prove a number of results about the WRT-TQFT \cite{A2,A3,A4,A5,AH,A6, AHJMMc}. 

In the case of weakly restricted familie, we cannot expect to prove projective flatness in general for such families due to the No-Go theorem of \cite{GM}. It is however very natural to ask for the existence of a Hitchin connection given by differential operators and not just by Toeplitz operators as the $L^2$-induced connection would be. We also expect that the Hitchin connection in some sense minimise the possible curvature, something which we hope to return to in our future work.

In the final two example sections we illustrate the applicability of our construction. We show that our construction applies to certain open subsets of the entire family of all complex structures on ${\mathbb R}^{2n}$ with the standard symplectic structure and certain open subsets of  the entire family of all complex structures on ${\mathbb R}^{2n}/{\mathbb Z}^{2n}$ again with the standard symplectic structure.  Further, our construction also applies to certain open subsets of the entire family of all complex structures on co-adjoint orbits and on the moduli spaces of flat $SU(n)$-connections on a surface of genus $g>1$, possibly with central holonomy around a point on the surface, with the Seshadri-Atiyah-Bott-Goldman symplectic structure on it.

\section{Quantization}
\label{sec:quantization}
In this chapter we will rather briefly introduce the mathematical theory of quantization and explain some of the problems that arise, when we try to define a mathematically rigid theory of quantization. One of the main points for us is the need to choose a polarization, which in our case wil be a Kähler structure compatible with the symplectic form on the manifold. This choice is not canonical and is auxiliary to the physical theory, and therefore we would suspect that the theory should, in some sense, be independent of this choice. The Hitchin connection aims to relate these different choices.

A quantization scheme is in the simplest form, a way to pass from classical mechanics to quantum mechanics. That is, to a system in classical mechanic, in the form of a phase space consistenting of a symplectic $2n$ dimensional manifold $(M, \omega)$, it assigns a corresponding Hilbert space $\mathcal{H}$ of quantum states, and to a classical observable given by a smooth function $f\in C^\infty(M)$ it assigns a self adjoint operater $Q(f)$ on $\mathcal{H}$.

For the theory to be physically sound, this assignment should be linear, send the constant function $1$ to the identity operator, and it should fullfill the commutation relation
\begin{equation}
  \label{eq:15}
   [Q(f),Q(g)] = i \hbar Q(\set{f,g}).
\end{equation}
Lastly, applying the quantization to $R^{2n}$ with the standard symplectic form should yield the canonical quantization (see \cite{W}). It has however been shown, that such a full quantization can't exit. The approach we will follow to handle this is to weaken the relation \eqref{eq:15}, so we only require the assymptotic relation
\begin{equation}
  \label{eq:20}
  [Q(f),Q(g)] = ih Q(\set{f,g}) + O(h^2) \quad \text{as} \quad h\to 0.
\end{equation}
We will only consider one type of quantization, namely geometric quantization.

\subsection{Prequantization}
\label{sec:prequantization}
The first step in geometric quantization is the prequantization. Here we construct a Hilbert space of quantum observables as sections of tensor powers of a so called \emph{prequantum line bundle} over the phase space $(M, \omega)$.
\begin{definition}[Prequantum line bundle]
  \label{def:3}
  A prequantum line bundle over the symplectic manifold $(M,\omega)$ is a triple $(\mathcal{L}, h, \cv{})$ consisting of a line bundle $\mathcal{L}$ over $M$ with a Hermitian metric $h$ and a compatible connection $\cv{}$, such that the curvature of $\cv{}$ is related to the symplectic form by the relation
\begin{equation}
  \label{eq:1}
  F_{\cv{}}= -i \omega.
\end{equation}
We say that a connection is compatible with the Hermitian structure $h$, if we have
\begin{equation}
  \label{eq:44}
  X h(s_1,s_2) = h(\cv{X}s_1,s_2) + h(s_1,\cv{X}s_2)
\end{equation}
for any vector field $X$ on $M$ and any two sections $s_1,s_2 \in C^\infty(M, \mathcal{L})$, and we call $(M,\omega)$ \emph{prequantizable}, if there exist a prequantum line bundle over it. 
\end{definition}
By looking at the real first chern class $\tilde{c_1}(\mathcal{L})$, it is seen that the condition \eqref{eq:1} is actually a restrictive condition, since $\tilde{c_1}(\mathcal{L})=\frac{i}{2\pi}[F_{\cv{}}]=[\frac{\omega}{2\pi}]$, which is not the case for all symplectic manifolds. It is however true, that $(M, \omega)$ is prequantizable, precisely when
\begin{equation*}
  \left[\frac{\omega}{2\pi}\right] \in \ima{( H^2(M, \Z) \to H^2(M, \bR) )},
\end{equation*}
which is the so-call \emph{prequantum condition}.
Now given a prequantum line bundle over $M$, we can define the \emph{prequantum space}. Here we remark, that a prequantum structure on a line bundle induces a prequantum structure on any tensorpower of the bundle, and we will also use $h$ and $\cv{}$ for the induced metric and connection.
\begin{definition}[Prequantum space]
  \label{def:4}
  The prequantum space of level $k\in\N$ is the infinite dimensional vector space of sections of the k'th tensor power of $\mathcal{L}$
  \begin{equation}
    \label{eq:2}
    \mathcal{H}^k = C^\infty(M, \mathcal{L}^k).
  \end{equation}
More precisely, we actually consider the $L^2$-completion with respect to the Hermitian inner product  on compactly supported sections, given by
\begin{equation}
  \label{eq:3}
  \iprod{s_1, s_2}
  =\frac{1}{n!}\int_M h(s_1, s_2) \omega^n.
\end{equation}
We will not distinguish between $\mathcal{H}^k$ and the completion in the following.
\end{definition}
Next we define the prequantum map, sending a function $f\in C^\infty(M)$ to an operator on $\mathcal{H}^k$ by the expression
\begin{equation}
  \label{eq:4}
  P_k(f)= -\frac{1}{k}\cv{X_f}+if,
\end{equation}
and with this definition, the prequantum map satisfies 
\begin{equation}
  \label{eq:5}
   [P_k(f),P_k(g)] = \frac{i}{k}P_k(\set{f,g}),
\end{equation}
and thus prequantization satisfies almost all the requirements for a quantization scheme. It does however fail to reproduce canonical quantization, since the prequantum space gives wave functions that depend independently on position and momentum. These are however not in line with canonical quantization, since we have twice as many degrees of freedom, as we are supposed too.

\subsection{Kähler Quantization}
\label{sec:polarization}
To deal with the problem, that prequantization produces a Hilbert space of twice the desired dimension, we introduce a polarization. We will only consider the case, where $(M, \omega)$ admits a compatible complex structure $J$ making $(M,\omega, J)$ Kähler and denote the Kähler manifold $M_J$. In this case we get complex structures on $\mathcal{L}^k$ given by $\delbar = \cv{}^{(0,1)}$, since $\omega$ has type $(1,1)$ with respect to $J$ and the prequantum condition thus ensures that $(F_{\cv{}})^{(0,2)}=0$. This means we can choose the quantum space to be the holomorphic sections
\begin{equation}
  \label{eq:22}
  H^{(k)}_J=H^0(M_J, \mathcal{L}^k)=\set{s\in \mathcal{H}^{(k)} \mid \cv{J}^{0,1}s=0}.
\end{equation}
This is a subspace of the prequantum space $\mathcal{H}^k$, and it is finite dimensional, if $M$ is compact. Now the prequantum operators do not in general preserve the subspace $H^{(k)}_J$, but since it is in fact a closed subspace, we have a projection map $$\pk_J\colon \mathcal{H}^k \to H^{(k)}_J,$$ and we just define the quantum operator by taking the prequantum operator and projecting the result back on $H^{(k)}_J$, that is
\begin{equation}
  \label{eq:23}
  Q_k(f)_J = \pk_J \circ P_k(f).
\end{equation}
With this construction we lose the commutation relation \eqref{eq:15}, but it can be shown that we still have the relation \eqref{eq:20} (see \cite{BMS, Sch3, Sch, Sch1, Sch2, Tuyn,KS}).

\section{Families of Kähler Structures}
\label{sec:famil-kahl-struct}
In this section we explore the properties of families of Kähler structures on a symplectic manifold $(M, \omega)$. We start out with a smooth family of integrable almost complex structures, all compatible with $\omega$, giving us a family of Kähler structures on $M$.

\subsection{Smooth Families of Kähler Structures}
\label{sec:smooth-famil-kahl}

We let $\mathcal{T}$ be a smooth manifold and $(M, \omega)$ a symplectic manifold. Then we say that $\mathcal{T}$ smoothly parametrizes a family of almost complex structures on $(M,\omega)$, if there exist a smooth map
\begin{equation}
  \label{eq:24}
  J\colon \mathcal{T} \to C^\infty(M, \End(TM)) \quad \text{mapping} \quad \sigma \mapsto J_\sigma
\end{equation}
such that $J_\sigma$ is an almost complex structure for each $\sigma\in \mathcal{T}$. We say that $J$ is smooth, when it defines a smooth section of the pullback bundle $\pi^*_M(\End (TM))$, where $\pi_M\colon \mathcal{T}\times M \to M$ is the projection on $M$.

We will look at the case where all $J_\sigma$ are integrable and compatible with the symplectic structure, such that $(M, \omega, J_\sigma)$ is Kähler for each $\sigma \in \mathcal{T}$.

A complex structure gives a splitting of the complexified tangent bundle, so now we have a splitting
\begin{equation*}
  TM_\bC = T'M_\sigma \oplus T''M_\sigma,
\end{equation*}
for each $\sigma \in \mathcal{T}$, into the holomorphic and anti-holomorphic parts, i.e. the $i$ and $-i$ eigenspaces of $J_\sigma$. 

 We denote the Kähler metric by $g$, and it is given by 
\begin{equation*}
  g=\omega \cdot J,
\end{equation*}
where the dot denotes contraction of tensors placed next to each other. We remark that for a vector $X$, we have the standard notation $i_X\omega=X\cdot \omega$. We will often need contraction of tensors and in more complicated expressions, we can't always indicate contraction by placing the tensors next to each other. For this, we will use abstract index notation to denote the entries of the tensor, and following the Einstein summation convention, repeated indices are contracted. In the same spirit, we use substript indices for covariant entries and superscript indices for contravariant entries. The contraction correspond to a summation in local coordinates, but our indices only name the entries of the tensor and do not represent a choice of local coordinates. Writing the expression from before in abstract index notation, would look like
\begin{equation*}
  g_{ab}=\omega_{au}J_b^u ,
\end{equation*}
and we remark, that we try to use the letters $a,b,c,d$ for entries, that are not contracted, and letters $u,v,w,x,y$ for contracted indices.

We will need the inverses of the metric and the symplectic form. These are the symmetric bivectorfield $\tilde{g}$ and antisymmetric bivectorfield $\tilde{\omega}$, such that 
\begin{equation*}
  g\cdot \tilde{g} = \tilde{g}\cdot g = \Id
  \quad \text{and} \quad
  \omega \cdot \tilde{\omega} = \tilde{\omega}\cdot \omega = \Id.
\end{equation*}
These exist, since the metric and symplectic forms are nondegenerate. We will sometimes use $g$ and $\tilde{g}$ to either lower or raise an index by contraction. It is also useful to record, that the relation between $g$ and $\omega$ implies that 
\begin{equation*}
  \tilde{g}= -J \cdot \tilde{\omega}.
\end{equation*}

Another construction, that we will need later, is the canonical line bundle of $M_J$, which is just defined as the top exterior power of the holomorphic tangent bundle
\begin{equation}
  \label{eq:35}
  K_J=\bigwedge\nolimits^m T'M_J.
\end{equation}
The hermitian structure on $T'M_J$ induced by $g$ again induces a hermitian structure in the canonical line bundle, which we will also just denote by $h$.

Now we want to investigate the variation of the family $J$. More precisely we will differentiate along a vectorfield $V$ on $\mathcal{T}$. This derivative is again a map into the endomorphism bundle which we denote by
\begin{equation*}
  V[J]\colon \mathcal{T} \to C^\infty(M, \End(TM)).
\end{equation*}
Differentiating the equality $J^2= -1$ and using the Leibniz rule, we get
\begin{equation*}
  JV[J] =  - V[J]J,
\end{equation*}
which shows that $V[J]_\sigma$ interchanges types on $TM_\C$, sending $i$ eigenvectors to $-i$ eigenvectors and the other way around. Thus $V[J]$ decomposes as
\begin{equation}
  \label{eq:31}
  V[J]=V[J]' + V[J]''
\end{equation}
where the two components
\begin{align*}
  &V[J]'_\sigma \in C^\infty(M, T''M^*_\sigma \otimes T'M_\sigma) \quad \text{and}\\
  &V[J]''_\sigma \in C^\infty(M, T'M^*_\sigma \otimes T''M_\sigma)
\end{align*}
are each others conjugates. Now since contraction in the first entry of $\omega$ defines an isomorphism $i_\omega \colon TM_\C \to TM_\C^*$, we can get any element in
\begin{equation*}
  C^\infty(M, \End(TM_\C)) =C^\infty (M,TM_\C\otimes TM_\C^*)
\end{equation*}
by contraction with a bivectorfield. We let $\tilde{G}(V)\colon \mathcal{T} \to C^\infty(M, TM_\C\otimes TM_\C)$ be such that the equality
\begin{equation}
  \label{eq:32}
  V[J]=\tilde{G}(V)\cdot \omega
\end{equation}
holds at each $\sigma \in \mathcal{T}$ and for each vectorfield $V$ on $\mathcal{T}$. We get another expression for $\tilde{G}(V)$ by differentiating the equality $\tilde{g}=-J\cdot\tilde{\omega}$ along $V$, namely
\begin{equation}
  \label{eq:33}
    V[\tilde{g}]=-V[J]\cdot\tilde{\omega} = -\tilde{G}(V).
\end{equation}
This is again using the Leibniz rule, and that $\tilde{\omega}$ does not depend on $\sigma$, so the derivative along any vector field $V$ vanishes. Since $\tilde{g}$ is symmetric, \eqref{eq:33} implies that $\tilde{G}(V)$ is also symmetric.

Looking at the types of $V[J]$ and $\omega$, we see that $\tilde{G}(V)$ has no $(1,1)$-part, and so we get a decomposition of $\tilde{G}(V)$ into
\begin{equation*}
  \tilde{G}(V)=G(V) + \bar{G}(V),
\end{equation*}
where
\begin{equation*}
  G(V)_\sigma \in C^\infty(M,S^2(T'M_\sigma)) \quad \text{and} \quad
  \bar{G}(V)_\sigma \in C^\infty(M,S^2(T''M_\sigma)).
\end{equation*}
Observe also that $\bar{G}(V)$ is actually the conjugate of $G(V)$, since $\tilde{G}(V)$ is real and thus its own conjugate.
We can also express the variation of the Kähler metric in terms of $\tilde{G}(V)$ by differetiating the compatibility condition $g=\omega \cdot J$, getting
\begin{align*}
  V[g]
  &=\omega \cdot V[J]\\
  &=\omega \cdot \tilde{G}(V)\cdot \omega\\
  &= g \cdot J \cdot \tilde{G}(V)\cdot g \cdot J\\
  &= - g \cdot J \cdot (G(V) + \bar{G}(V)) \cdot J \cdot g\\
  &= - i^2 g \cdot G(V) \cdot g  - ( -i)^2 g \cdot \bar{G}(V) \cdot g\\
  &=  g \cdot \tilde{G}(V) \cdot g.
\end{align*}
This also shows that $V[g]\in C^\infty(M, S^2(T'M_\sigma^*)\oplus S^2(T''M_\sigma^*))$. One more relation about variations, that we will need, is the variation of the Levi-Civita connection $\nabla^g$, here superscripted with $g$ to denote that it is the connection related to the metric $g$. We will not calculate this here, but just state the result, which is given implicitly in (\cite{Besse} Theorem 1.174) by the formula
\begin{equation*}
  2g(V[\nabla^g]_XY,Z)
  =\cv{X}(V[g](Y,Z)) + \cv{Y}(V[g])(X,Z) - \cv{Z}(V[g])(X,Y)
\end{equation*}
for vectorfields $X,Y,Z$ on $M$. To give an explicit expression, we use the above result and write it in the index notation as
\begin{equation}
  \label{eq:34}
  2V[\nabla^g]_{ab}^c
  =\cv{a}\tilde{G}(V)^{cu}g_{ub} + g_{au}\cv{b}\tilde{G}(V)^{uc}  -g_{au}\tilde{g}^{cw}\cv{w}\tilde{G}(V)^{uv}g_{vb}.
\end{equation}

\subsection{The Canonical Line Bundle of a Family}
\label{sec:canon-line-bundle}
Before we start the calculations in this section, we recall that the Ricci form is the skew-symmetric $(1,1)$-form $\rho=J\cdot r$, where $r$ is the Ricci curvature tensor, which is given by
\begin{equation*}
  r(X,Y)=\tr (Z \mapsto R(Z,X)Y) \quad \text{or in index notation} \quad r_{ab}=R_{wab}^w,
\end{equation*}
where $R$ is the curvature of the Levi-Civita connection. We also recall that the Levi-Civita connection and the curvature $R(X,Y)$ preserves types on $TM_\bC$. The form part of $R$ is $J$-invariant and is of type $(1,1)$.

Our purpose in this section is to derive an expression for the variation of the Ricci form $\rho$. This will not seem apparent from the beginning, but we will construct a certain line bundle over the product manifold $\mathcal{T}\times M$ and consider the induced connection in this bundle. The Ricci form will appear in an expression for the curvature in some directions on $\mathcal{T}\times M$ and using the Bianchi identity, we will get a very useful relation.

First we consider the vector bundle $\hat{T}'M \to \mathcal{T}\times M$, where the fibers are just the holomorphic tangent spaces of M, that is $\hat{T}'M_{(\sigma,p)}=T_p'M_\sigma$. So the point $\sigma \in \mathcal{T}$ determines the splitting of the bundle $TM\otimes{\mathbb C}$, and the the point $p\in M$ chooses the fiber $T'_pM_\sigma$ of the sub-bundle $T'M_\sigma$. We use the hat in the notation to denote, that we consider it as a bundle over the product $\mathcal{T}\times M$, and similarly we will use a hat on the differential $\hat{d}$ and connection $\hat{\nabla}$ on this bundle.

We notice that the Kähler metric induces a Hermitian structure $\hat{h}$ on $\hat{T}'M$.  We construct a connection on $\hat{T}'M$ in two steps. We first notice that along vector fields on $M=\set{\sigma}\times M$, we can use the Levi-Civita Connection on the bundle $T'M_\sigma$, which gives us a partial connection on $\hat{T}'M$ compatible with the Hermitian structure.

Now we think of a section $Z\in C^\infty(\mathcal{T}\times M,\hat{T}'M)$ as a smooth family of sections of the holomorphic tangent bundles, and we let $V$ be a vector field on $\mathcal{T}=\mathcal{T}\times \set{p}$. Since each of the holomorphic tangent spaces sits inside the larger complexified tangent bundle $T'M_\sigma \subset TM\otimes{\mathbb C}$, we can think of $Z$ as a smooth family of vector fields in this bundle. Since $TM\otimes{\mathbb C}$ is unchanged along $V$, we can take the variation of $Z$ along $V$ in $TM\otimes{\mathbb C}$, and then project the result back onto the holomorphic sub-bundle. This defines a connection along the directions on $\mathcal{T}$, that is
\begin{equation*}
  \hat{\cv{}}_VZ = \pi^{1,0}V[Z].
\end{equation*}

We still want this connection to be compatible with the Hermitian structure, so we check this by calculation. So let $V$ be a vector field on $\mathcal{T}$ and $X,Y$ sections of $\hat{T}'M$. Then we have
\begin{align*}
  V[\hat{h}(X,Y)]
  =V[g(X,\bar{Y})]
  &=V[g](X,\bar{Y}) + g(V[X],\bar{Y}) + g(X,\overline{V[Y]})\\
  &=g(V[X],\bar{Y}) + g(X,\overline{V[Y]})\\
  &=h(\hat{\nabla}_VX,Y) + h(X,\hat{\nabla}_VY),
\end{align*}
since the $(1,1)$-part of $V[g]$ vanishes as shown earlier. This is exactly the condition, that $\hat{\nabla}$ is compatible with the Hermitian structure. In this way we have constructed a connection on all of $\hat{T}'M$ compatible with the Hermitian structure.

Next we consider the top exterior power, which as in the case of $T'M_\sigma$ gives the \emph{canonical line bundle of the family} of complex structures, and we will denote this
\begin{equation}
  \label{eq:36}
  \hat{K}
  =\bigwedge\nolimits^m\hat{T}'M^*
  \to \mathcal{T}\times M.
\end{equation}
Just as for the normal canonical line bundle, we get an induced Hermitian structure and a compatible connection on $\hat{K}$.

We will now state the proposition from \cite{AGL} about the curvature of $\hat{\nabla}^K$. 
\begin{proposition}
  \label{pro:5}
  Given vector fields $X,Y$ on $M$ and $V$ on $\mathcal{T}$, the curvature of $\hat{\nabla}^K$ is given by
\begin{align}
  \label{eq:37}
  &F_{\hat{\nabla}^K}(X,Y) = i \rho(X,Y),\\
  &F_{\hat{\nabla}^K}(V,X) = \frac{i}{2} \delta \tilde{G}(V)\cdot \omega\cdot X
\end{align}
\end{proposition}
We remark that we can also get an expression for the curvature for two vector fields $V,W$ on $\mathcal{T}$, but we will not need this result here. Now applying the Bianchi identity and the results of proposition \ref{pro:5} gives us the desired result about the variation of the the Ricci form. 
\begin{proposition}
  \label{pro:7}
  The variation of the Ricci form along a vector field $V$ on $\mathcal{T}$ is given by
\begin{equation}
  \label{eq:29}
  V[\rho]= \frac{1}{2}d(\delta \tilde{G}(V)\cdot \omega).
\end{equation}
\end{proposition}
\begin{proof}
  Let $X,Y$ be commuting vector fields on $M$ and $V$ a vector field on $\mathcal{T}$. Then the Bianchi identity for $\hat{\nabla}^K$ gives
\begin{align*}
  0
  &=V[F_{\hat{\nabla}^K}(X,Y)]
    +X[F_{\hat{\nabla}^K}(Y,V)]
    +Y[F_{\hat{\nabla}^K}(V,X)]\\
  &=i V[\rho(X,Y)]
    -X\left [\frac{i}{2} \delta \tilde{G}(V)\cdot \omega\cdot Y \right]
    +Y\left [\frac{i}{2} \delta \tilde{G}(V)\cdot \omega\cdot X \right]\\
&=i V[\rho(X,Y)]
  -\frac{i}{2} d(\delta \tilde{G}(V)\cdot \omega) (X,Y),
\end{align*}
where the last equality follows by the invariant formula for the exterior derivative, since $X$ and $Y$ were chosen to commute. Now isolating $V[\rho(X,Y)]$ gives the desired equality.
\end{proof}

\subsection{Holomorphic Families of Kähler Structures}
\label{sec:holom-famil-kahl}
In this and the next section we will introduce some extra restriction on a family of complex structures. Assume that $\mathcal{T}$ is a complex manifold, and then we can require, that the the family $J$ of holomorphic structures define a holomorphic map from $\mathcal{T}$ to the space of holomorphic structures. This is defined as follows.
\begin{definition}
  \label{def:10}
  Let $\mathcal{T}$ be a complex manifold and $ J\colon \mathcal{T} \to C^\infty(M, \End(TM))$ a family of integrable almost complex structures on $M$. Then $J$ is holomorphic if
\begin{equation*}
  V'[J]=V[J]' \quad \text{and} \quad
  V''[J]=V[J]''
\end{equation*}
for any vector field $V$ on $\mathcal{T}$.
\end{definition}
Now lets denote the integrable almost complex structure on $\mathcal{T}$ by $I$. Then we get an almost complex structure $\hat{J}$ on the product manifold $\mathcal{T}\times M$ defined by 
\begin{equation*}
  \hat{J}(V\oplus X)= IX \oplus J_\sigma X, \quad \text{for} \quad
  V\oplus X \in T_{(\sigma,p)}(\mathcal{T}\times M).
\end{equation*}
It can be shown that holomorphicity of the family of complex structure as defined above is equivalent to the $\hat{J}$ being an integrable almost complex structure on $\mathcal{T}\times M$ (see \cite{AG1}). This is shown by using the criterion of the vanishing of the Nijunhuis tensor.

A useful consequence of holomorphicity is that 
\begin{equation}
  \label{eq:43}
  \tilde{G}(V')=V'[J] \cdot \tilde{\omega} = V[J]' \cdot \tilde{\omega} = G(V),
\end{equation}
and  similarly $\tilde{G}(V'')=\bar{G}(V)$.

\section{The Hitchin Connection}
\label{sec:hitchin-connection}
In this section we give the construction of the Hitchin Connection. The ideas of the proofs follow the original construction by the first author in \cite{A5}, though some things are done with a slightly different touch, which in large parts are inspired by \cite{AG1}, which is joined work by the first author and Gammelgaard.

The theorem and the setup is stated in the introduction, theorem \ref{thm:2}, and as in the rigid setting the proof is done by constructing a $u$, which satisfies the condition in the following lemma.
\begin{lemma}
  \label{lem:3}
  The connection $\HC=\cv{}^T + u$ is a Hitchin connection if and only if 
  \begin{equation}
    \label{eq:8}
    \cv{\sigma}^{0,1}u(V)s
    =\frac{i}{2} V'[J_\sigma]\cdot \cv{\sigma}^{1,0}s,
  \end{equation}
  for any holomorphic section $s\in H^{(k)}_\sigma$ and any smooth vector field $V$ on $\mathcal{T}$.
\end{lemma}
\begin{proof}
  By assumption we need to have
  \begin{align*}
    0&=\cv{\sigma}^{0,1}(\HC_{V}s)\\
     &=\cv{\sigma}^{0,1}V[s] + \cv{\sigma}^{0,1} u(V)s.
  \end{align*}
  Now by differentiating $\cv{\sigma}^{0,1}s=0$ along $V$ we get
  \begin{align*}
    0&=V[\cv{\sigma}^{0,1}s]= V\left[ \frac{1}{2} (Id + i J_\sigma) \cdot \cv{}s\right]\\
     &=\frac{i}{2} V[J_\sigma]\cdot \cv{\sigma}^{1,0}s  + \cv{\sigma}^{0,1}V[s]
  \end{align*}
  Combining the above expressions, we get the equation \eqref{eq:8}.
\end{proof}

The construction of the connection is carried out through a number of lemmas. We will start in the most general setting and then add the assumptions in the lemmas, when we need them. Firstly we just assume that we have a symplectic manifold $(M, \omega)$ with a prequantum line bundle $\mathcal{L}$, a family of Kähler structures $J$, and that we have an arbitrary symmetric bivector field
\begin{equation*}
  G \in C^\infty(M_\sigma , S^2(T'M_\sigma)),
\end{equation*}
and then we get a linear bundle map
\begin{equation*}
  G \colon T'M_\sigma^* \to T'M_\sigma,
\end{equation*}
given by contracting with $G$. Now using this we get an operator $\lp{G}$ on $\mathcal{H}^{(k)}$ given by
\begin{align*}
  \lp{G} \colon
  \mathcal{H}^{(k)}
  = C^\infty(M, \mathcal{L}^k)
  &\xrightarrow{\cv{\sigma}^{1,0}}  C^\infty(M, T'M_\sigma^* \otimes \mathcal{L}^k)\\
  &\xrightarrow{G\otimes \Id}  C^\infty(M, T'M_\sigma \otimes \mathcal{L}^k)\\
  &\xrightarrow{\cv{\sigma}^{1,0}\otimes \Id + \Id \otimes \cv{\sigma}^{1,0}}
  C^\infty(M, T'M_\sigma^* \otimes T'M_\sigma \otimes \mathcal{L}^k)\\
  &\xrightarrow{\tr{}}  C^\infty(M, T'M_\sigma \otimes \mathcal{L}^k).
\end{align*}
In abstract tensor notation we can write this in the short form
\begin{equation*}
  \lp{G}s=\cv{u'}G^{u'v'}\cv{v'}s,
\end{equation*}
where the outer connection is the connection in the tensor product $T'M_\sigma \otimes \mathcal{L}^k$, which is given exactly as described above by the Leibniz rule.
\begin{lemma}
  \label{lem:4}
  Let $G \in H^0(M_\sigma , S^2(T'M_\sigma))$ be any holomorphic bi-vector field on $(M, \omega_\sigma)$, then
  we have for any section $s\in H_\sigma^{(k)}$
  \begin{equation}
    \label{eq:9}
    \cv{}^{(0,1)}\lp{G}s
    = -2ik \omega \cdot G\cdot \cv{}^{(1,0)}s - i\rho \cdot G \cdot \cv{}^{(1,0)}s -ik\omega\cdot\delta(G) s.
  \end{equation}
\end{lemma}
\begin{proof}
  The proof is a calculation that mainly uses the trick of commuting two covariant derivatives to get one term that disappears because of type considerations plus a curvature term. We will write out the proof using abstract tensor notation, which highlights contraction of terms. So for $s\in H_\sigma^{(k)}$ we get that
  \begin{align*}
    \cv{}^{(0,1)}\lp{G}s
    &=\cv{a''}\cv{u'}G^{u'v'}\cv{v'}s \\
    &=\cv{u'}\cv{a''}G^{u'v'}\cv{v'}s +[\cv{}, \cv{}]_{a''u'}G^{u'v'}\cv{v'}s\\
    &=\cv{u'}G^{u'v'}\cv{a''}\cv{v'}s +[\cv{}, \cv{}]_{a''u'}(G^{u'v'})\cv{v'}s +G^{u'v'}[\cv{}, \cv{}]_{a''u'}\cv{v'}s\\
    &=\cv{u'}G^{u'v'}[\cv{},\cv{}]_{a'' v'}s +R_{a''wu'}^wG^{u'v'}\cv{v'}s  -ik\omega_{a''u'} G^{u'v'}\cv{v'}s\\
    &= -ik\cv{u'}G^{u'v'}\omega_{a'' v'}s -R_{wa''u'}^wG^{u'v'}\cv{v'}s  -ik\omega\cdot G \cdot \cv{}s\\
    &= -ik\cv{u'}(G^{u'v'})\omega_{a'' v'}s -ikG^{u'v'}\omega_{a'' v'}\cv{u'}s -r_{a''u'}G^{u'v'}\cv{v'}s  -ik\omega\cdot G \cdot \cv{}s\\
    &= -ik\omega_{a'' v'}\delta(G)^{v'}s -ik\omega \cdot G \cdot \cv{}s -J_a^xr_{xy}J^y_{u'}G^{u'v'}\cv{v'}s -ik\omega\cdot G \cdot \cv{}s\\
    &= -2ik\omega\cdot G \cdot \cv{}s -ik\omega \cdot \delta(G)s -i\rho_{au'}G^{u'v'}\cv{v'}s \\
    &= -2ik\omega\cdot G \cdot \cv{}s -i \rho \cdot G \cdot\cv{}s -ik\omega \cdot \delta(G)s.
  \end{align*}
\end{proof}
The next step in the construction is to use the assumption on the first chern class. Applying this we get the following.
\begin{corollary}
  \label{cor:1}
  Consider the situation as in lemma \ref{lem:4} and assume that the family admits a family of Ricci potentials $F$. Then we get
  \begin{equation}
    \label{eq:10}
    \cv{}^{(0,1)}\lp{G}s
    = -i(2k+n) \omega \cdot G\cdot \cv{}^{(1,0)}s + 2 (d \delbar_\sigma F_\sigma)\cdot G \cdot \cv{}^{(1,0)}s -ik\omega\cdot\delta(G) s.
  \end{equation}
\end{corollary}
\begin{proof}
  The proof follows directly by inserting the expression for the Ricci form given in terms of the family of Ricci potentials $F_\sigma$, i.e. $\rho_\sigma=n\omega_\sigma + 2i d \delbar_\sigma F_\sigma$, in \eqref{eq:9}.
\end{proof}
To get rid of the second term here, we can use the following proposition.
\begin{lemma}
  \label{lem:5}
  Under the same assumptions as above, we have that
  \begin{equation}
    \label{eq:11}
    \cv{}^{(0,1)}(\cv{G\cdot d F} s) = -ik\omega\cdot G \cdot dF s -(d \delbar_\sigma F_\sigma)\cdot G \cdot \cv{}^{(1,0)}s,
  \end{equation}
  and thus we get the equality
  \begin{equation}
    \label{eq:12}
    \cv{}^{(0,1)}(\lp{G}s + 2 \cv{G\cdot d F} s)
    = -i(2k+n) \omega \cdot G\cdot \cv{}^{(1,0)}s -ik\omega\cdot\delta(G)s -2 ik\omega\cdot G \cdot dF s.
  \end{equation}
\end{lemma}
\begin{proof}
  Again the proof is a calculation. For $s\in H_\sigma^{(k)}$, we have that
  \begin{align*}
    \cv{}^{(0,1)}(\cv{G\cdot d F} s)
    &=\cv{a''}G^{u'v'}dF_{v'}\cv{u'}s\\
    &=G^{u'v'}dF_{v'}\cv{a''}\cv{u'}s + G^{u'v'} \cv{a''}(dF_{v'})\cv{u'}s\\
    &=G^{u'v'}dF_{v'}[\cv{}, \cv{}]_{a''u'}s + \cv{a''}(dF_{u'}) G^{u'v'} \cv{v'}s\\
    &= -ik\omega_{a''u'} G^{u'v'}dF_{v'}s + (\delbar dF)\cdot G \cdot \cv{}s\\
    &= -ik\omega \cdot G \cdot dF s + (\delbar \del F)\cdot G \cdot \cv{}s\\
    &= -ik\omega \cdot G \cdot dF s - (\del \delbar F)\cdot G \cdot \cv{}s.
  \end{align*}
  Now \eqref{eq:12} follows by combining equations \eqref{eq:10} and \eqref{eq:11}.
\end{proof}

\subsection{The Weakly Restricted Case}
\label{sec:weakly-rigid-case}
Now we impose the weakly restricted condition and continue the construction, in the case where are family of complex structures is holomorphic. The following lemma and proposition are the key components.
\begin{lemma}
  \label{lem:1}
   Consider the situation as in corollary \ref{cor:1}, and assume that the family of holomorphic structures $J_\sigma$ is weakly restricted. Furthermore assume that the family admits a family of Ricci potentials $F$.

   Let $G_{\beta}(V)$ and $\beta(V)$ be the bivector- and vector fields associated to the the family $J_\sigma$. Then we have that the $1$-form
\begin{equation}
  \label{eq:14}
  \Omega(V)
    =-\delta(G_{\beta}(V))\cdot \omega +\delta(\delbar\beta(V)) - 2dF \cdot G_{\beta}(V)\cdot \omega + 2\delbar \beta(V) \cdot dF+ 4i\delbar V'[F]
\end{equation}
is closed and of type $(0,1)$.
\end{lemma}
\begin{proof}
  We start with the equation $\rho_\sigma = n\omega + 2id \delbar_\sigma F_\sigma$. Differentiating this equation along $V'$ we get
  \begin{align*}
    V'[\rho] 
    &= 2id V'[\delbar] F + 2id \delbar V'[F]\\
    &= - d V'[J] \cdot dF + d (2i\delbar V'[F])\\
    &= - d (G_{\beta}(V)\cdot \omega) \cdot dF + d\delbar \beta(V) \cdot dF + d (2i\delbar V'[F])\\
    &= - d (dF \cdot G_{\beta}(V)\cdot \omega) + d(\delbar \beta(V) \cdot dF) + d (2i\delbar V'[F]).
  \end{align*}
  Now we use proposition \ref{pro:7}, which gives us
  \begin{equation*}
    2 V'[\rho] 
    =d\delta(G(V)\cdot \omega)
    = d\delta(V'[J])
    = d\delta(G_{\beta}(V))\cdot \omega -d\delta(\delbar\beta(V)).
  \end{equation*}
  Inserting this on the left hand side of the above equation, we get
  \begin{eqnarray*}
    0 & = & -d(\delta(G_{\beta}(V))\cdot \omega) +d\delta(\delbar\beta(V))
    - d (2dF \cdot G_{\beta}(V)\cdot \omega)\\
    & & + d(2\delbar \beta(V) \cdot dF) + d (4i\delbar V'[F]) = d\Omega(V),
  \end{eqnarray*}
which exactly states that $\Omega(V)$ is closed. By checking each term, it is also seen directly to be of type $(0,1)$.
\end{proof}

\begin{proposition}
  \label{pro:1}
  Consider the setup of lemma \ref{lem:1}, and furthermore assume that $H^1(M, \bR)=0$. Then there exists $\psi(V)\in C^\infty(M)$, such that
  \begin{align*}
    &\delta(G_{\beta}(V)) \cdot \omega + 2 dF \cdot G_{\beta}(V) \cdot \omega\\
    &=4i \delbar V'[F] +2\delbar  (d F \cdot \beta(V))
    + in\omega \cdot \beta(V) + \delbar\delta(\beta(V)) -\delbar\psi(V).
  \end{align*}
  If non of the complex structures admit non-constant holomorphic functions on $M$, which is true for instance if $M$ is compact, we get that $\psi(V)=0$.
\end{proposition}
\begin{proof}
  We know that $\Omega(V)$ is closed and of type $(0,1)$, and since we have assumed $H^1(M, \bR)=0$, it is exact. Thus there exist a function $\psi(V) \in C^\infty(M)$, such that $\Omega(V)=\delbar\psi(V)$. 

  Observe that if non of the complex structures admit non-constant holomorphic functions on $M$, we get that $\psi(V)=0$, since the equation $d\psi(V)=\delbar\psi(V)$ shows that it is anti-holomorphic.

Combining expressions we get
  \begin{equation*}
    \delta(G_{\beta}(V))\cdot \omega + 2dF \cdot G_{\beta}(V)\cdot \omega
    =4i\delbar V'[F] + 2 dF \cdot \delbar \beta(V)+\delta(\delbar\beta(V)) -\delbar\psi(V).
  \end{equation*}
  Now we only need to rewrite the term $\delta(\delbar\beta(V))$, and this is again done by an application of commuting covariant derivatives, which goes as follows
  \begin{align*}
    \delta(\delbar\beta(V))
    &=\cv{u'}\cv{a''}\beta(V)^{u'}\\
    &=[\cv{},\cv{}]_{u' a''}\beta(V)^{u'} + \cv{a''}\cv{u'}\beta(V)^{u'}\\
    &=R_{u' a''v'}^{u'}\beta(V)^{v'} +\delbar \delta(\beta(V))\\
    &=r_{a''v'}\beta(V)^{v'} +\delbar \delta(\beta(V))\\
    &=J_{a''}^{u}r_{uw}J^{w}_{v'}\beta(V)^{v'} +\delbar \delta(\beta(V))\\
    &=i \rho_{a''v'}\beta(V)^{v'} +\delbar \delta(\beta(V))\\
    &=i \rho \cdot \beta(V) +\delbar \delta(\beta(V))\\
    &=i n\omega \cdot \beta(V) - 2\del\delbar F \cdot \beta(V)+\delbar \delta(\beta(V)\\
    &=i n\omega \cdot \beta(V) + 2\delbar d F \cdot \beta(V)+\delbar \delta(\beta(V).
  \end{align*}
  Inserting this above and rewriting we obtain that
  \begin{equation*}
    \delbar dF\cdot \beta(V) + d F \cdot \delbar\beta(V))
    =\delbar (d F \cdot \beta(V)),
  \end{equation*}
which  completes the proof.
\end{proof}
We can now insert this in the expression from lemma \ref{lem:5} and get
\begin{align*}
  \cv{}^{(0,1)}&(\lp{G_{\beta}(V)}s + 2 \cv{G_{\beta}(V)\cdot dF} s)\\
                &=-i(2k+n) \omega \cdot G_{\beta}(V)\cdot \cv{}^{(1,0)}s
                  -ik(\omega\cdot\delta(G_{\beta}(V))s +2 \omega\cdot G_{\beta}(V) \cdot dF s)\\
                &= i(2k+n)  (G_{\beta}(V)\cdot\omega) \cdot \cv{}^{(1,0)}s 
                  +ik(\delta(G_{\beta}(V))\cdot \omega +2 dF \cdot G_{\beta}(V) \cdot \omega)s\\
                &=  i(2k+n) (V'[J] + \delbar \beta(V)) \cdot \cv{}^{(1,0)}s \\
                &+ik(4i \delbar V'[F] +2\delbar  (dF \cdot \beta(V))
                  + in\omega \cdot \beta(V) + \delbar\delta(\beta(V)) -\delbar\psi(V))s\\
                &= 2(2k+n) \frac{i}{2} V'[J] \cdot \cv{}^{(1,0)}s  + i(2k+n)\delbar \beta(V) \cdot \cv{}^{(1,0)}s\\
                &-4k \delbar V'[F]s +2ik\delbar  (dF \cdot \beta(V))s
                  -kn\omega \cdot \beta(V)s + ik\delbar\delta(\beta(V))s -ik\delbar\psi(V)s.
\end{align*}
Now we only need one last lemma to get rid of the last first-order term on the right side.
\begin{lemma}
  \label{lem:7}
  For any family of vector fields $\beta(V) \in C^\infty(M_\sigma , T'M_\sigma)$ and a holomorphic section $s\in H_\sigma^k$, we have that
\begin{equation}
  \label{eq:7}
  \cv{}^{(0,1)} \cv{\beta(V)} s
  =\delbar \beta(V)\cdot \cv{}^{(1,0)}s -ik\omega\cdot \beta(V) s
\end{equation}
\end{lemma}
\begin{proof}
  The result follows directly by the following calculation
\begin{align*}
  \cv{}^{(0,1)} \cv{\beta(V)} s
  &=\delbar \beta(V) \cdot \cv{}^{(1,0)}s + \beta(V)^{u'}\cv{a''}\cv{u'}s \\
  &=\delbar \beta(V) \cdot \cv{}^{(1,0)}s + \beta(V)^{u'}[\cv{},\cv{}]_{a'' u'} s \\
  &=\delbar \beta(V) \cdot \cv{}^{(1,0)}s - \beta(V)^{u'} ik\omega_{a'' u'} s \\
  &=\delbar \beta(V)\cdot \cv{}^{(1,0)}s -ik\omega\cdot \beta(V) s.
\end{align*}
\end{proof}
Now using this lemma we get that
\begin{align*}
  \cv{}^{(0,1)}&(\lp{G_{\beta}(V)}s + 2 \cv{G_{\beta}(V)\cdot dF}s-i(2k+n) \cv{\beta(V)}s )\\
               &= 2(2k+n) \frac{i}{2} V'[J] \cdot \cv{}^{(1,0)}s -4k \delbar V'[F]s\\
               &+2ik\delbar  (dF \cdot \beta(V))s - 2k(k+n) \omega \cdot \beta(V)s + ik\delbar\delta(\beta(V))s -ik\delbar\psi(V)s,
\end{align*}
and now by moving all the 0'th order terms to the left side, we our result. Here $\phi(V) \in C^\infty(M)$ is a smooth function, such that $\delbar \phi(V)=\omega\cdot \beta(V)$, and thus we get that
\begin{align*}
  \cv{}^{(0,1)}&(\lp{G_{\beta}(V)}s + 2 \cv{G_{\beta}(V)\cdot dF}s-i(2k+n) \cv{\beta(V)}s \\
               &+4k V'[F]s -2ik dF\cdot \beta(V)s -ik\delta(\beta(V))s +2k(k+n)\phi(V)s +ik\psi(V)s)\\
               &= 2(2k+n) \frac{i}{2} V'[J] \cdot \cv{}^{(1,0)}s.
\end{align*}
Now we can complete the proof of theorem \ref{thm:2}, since we see that we get a Hitchin connection, by setting
\begin{align*}
  u(V)&=\frac{1}{2(2k+n)}(\lp{G_{\beta}(V)} + 2 \cv{G_{\beta}(V)\cdot dF}-i(2k+n) \cv{\beta(V)} \\
      &+4k V'[F] -2ik dF\cdot \beta(V) -ik\delta(\beta(V)) +2k(k+n)\phi(V) +ik\psi(V)).
\end{align*}

\subsection{Hitchin connection for smooth families of complex structures}
\label{sec:hitch-conn-smooth}
We have in the above constructions of the Hitchin connection assumed that the family of complex structures was holomorphic. We can however go through the construction without assuming that ${\mathcal T}$ is a complex manifold. We have not used holomophicity of the family before assuming rigidity/weakly restricted, so instead of differentiation along $V'$ in lemma 4, we instead differentiate along $V$.

Doing this we get that the form
\begin{align*}
      &-\delta(G_{\beta}(V))\cdot \omega - 2dF \cdot G_{\beta}(V)\cdot \omega +\delta(\delbar\beta(V)) + 2\delbar \beta(V) \cdot dF+ 4i\delbar V[F]\\
      &-\delta(\bar{G}(V))\cdot \omega - 2dF \cdot \bar{G}(V)\cdot \omega
\end{align*}
is closed and hence exact. It is however no longer of type $(0,1)$, but it splits into a $(1,0)$ and a $(0,1)$ part, which come as $\del$ and $\delbar$ of a function $\tilde{\psi}(V) \in C^\infty(M)$. Both of the new terms are of type $(1,0)$, so we get similarly as above
\begin{equation*}
  \delbar\tilde{\psi}(V)=
  -\delta(G_{\beta}(V))\cdot \omega +\delta(\delbar\beta(V))- 2dF \cdot G_{\beta}(V)\cdot \omega + 2\delbar \beta(V) \cdot dF+ 4i\delbar V[F].
\end{equation*}
Arguing as in the proof of proposition \ref{pro:1}, we get that
\begin{align*}
    & \delta(G_{\beta}(V)) \cdot \omega + 2 dF \cdot G_{\beta}(V) \cdot \omega\\
    &=4i \delbar V[F] +2\delbar  (d F \cdot \beta(V)) + in\omega \cdot \beta(V) + \delbar\delta(\beta(V)) -\delbar\psi(V).
  \end{align*}
Now going through the construction of the Hitchin Connection as above, still assuming weakly restricted but without holomorphicity of the family of Kähler structures, we get the following theorem.

\begin{theorem}[Hitchin Connection for smooth $\mathcal{T}$]
  \label{thm:3}
Consider the same setup as in theorem \ref{thm:2}, except the manifold $\mathcal{T}$ is only assumed to be smooth and the assumption of holomorphicity of the family $J$ is dropped.
  Then there exists a Hitchin connection $\HC$ in the bundle $\hat{H}^{(k)}$ over $\mathcal{T}$, given by the expression
 $$
    \HC_V= \cv{V}^T + u(V), 
    $$
where
\begin{align}
      u(V)&=\frac{1}{2(2k+n)}(\lp{G_{\beta}(V)} + 2 \cv{G_{\beta}(V)\cdot dF}-i(2k+n) \cv{\beta(V)} \nonumber\\
      &+4k V[F] -2ik dF\cdot \beta(V) -ik\delta(\beta(V)) +2k(k+n)\phi(V) + ik\tilde{\psi}(V)),
  \end{align}
where $\phi(V)$ is defined as in Theorem \ref{thm:2} and $\tilde{\psi}(V) \in C^\infty(M)$ satisfies
$$ \delbar \tilde{\psi}(V) = \tilde{\Omega}(V),$$
and $\tilde{\Omega}(V) \in \Omega^1(M)$ is the closed and hence exact $1$-form
\begin{equation*}
  \tilde{\Omega}(V)
  =-\delta(G_{\beta}(V))\cdot \omega +\delta(\delbar\beta(V))- 2dF \cdot G_{\beta}(V)\cdot \omega + 2\delbar \beta(V) \cdot dF+ 4i\delbar V[F].
\end{equation*}
\end{theorem}

\section{The no-go theorem and projective flatness}
\label{sec:no-go-theorem}
In this section we briefly recall the work \cite{GM}, in which Ginzburg and Montgomery shows a no-go theorem, stating conditions under which no natural projectively flat connection can exist on the vector bundle of quantizations. It turns out that the Hitchin connection constructed in this paper fullfils the conditions, and thus it cannot be projectively flat in general.

To state the theorem we need to fix notation. We let ${\mathcal H}$ be the group of hamiltonian symplectomorphisms of $M$, and ${\mathcal G}$ be the group of diffeomorphism of the unit circle bundle $U$ of $\mathcal{L}$ which preserve the connection form. Lastly we let ${\mathcal G}_0$ be the identity connected component of ${\mathcal G}$, that is the the elements isotopic to the $\Id$ in ${\mathcal G}$. Let $J_0\in {\mathcal C}_\omega(M)$ and let ${\mathcal C}^0_\omega(M)$ be a small enough neighbourhood of $J_0$, such that $\mathcal{H}^{(k)}\mid_{{\mathcal C}^0_\omega(M)}$ is a vector bundle over ${\mathcal C}^0_\omega(M)$.

\begin{theorem}[Ginzburg and Montgomery]
  \label{thm:1}
  Assume that there exist a complex structure $J_0$ with stabilizer $G_{J_0}$ in ${\mathcal H}$ of positive dimension, and that the infinitesimal representation of $G_{J_0}$ on $H^{(k)}_{J_0}$ is non-trivial. Then there is no projectively flat connection on $\mathcal{H}^{(k)}\mid_{{\mathcal C}^0_\omega(M)}$, which is invariant under the ${\mathcal G}_0$ local action.
\end{theorem} 

We will now consider an example where we can apply our construction for a certain small enough neighbourhood of a particular $J_0$ with such a symmetry group.

Let $G$ be a compact simple and simply-connected Lie group. We are going to consider a co-adjoint orbit $M$ in $\Lie{g}^*$. On $M$ we are going to consider the Kirillov-Kostant symplectic structure (see e.g.\cite{W}). Furthermore, we have the natural $G$-invariant complex structure $J_0$ on $M$ coming from the identification
$$ M =G^{\mathbb C}/P,$$
where $P$ is a parabolic subgroup determined $M$. It is well know that $(M,J_0)$ is rigid and that there exist a small enough neighbourhood ${\mathcal C}^0_\omega(M)$ of $J_0$ such that for 
$$H(M_J, T_J) = 0$$ 
for all complex structures $J\in {\mathcal C}^0_\omega(M)$. 

 We now want to determine $\beta(V)_J$ uniquely for all $J\in {\mathcal C}^0_\omega(M)$ and all $V\in T_J{\mathcal C}^0_\omega(M)$ solving
\begin{equation*}
  V'[J]_J=  -\delbar_J \beta(V)_J
\end{equation*}
This we can do uniquely by the above vanishing of $H(M_J, T_J) $ and if we impose suitable conditions on $\beta(V)_J$. One possible such is to require that $\beta(V)_J$ is orthogonal to all homomorphic vector fields on $(M,J)$. Another way could be to require special evaluation properties of $\beta(V)_J$ at various points on $M$. Further, we can determine a smooth family of Ricci potentials, by picking, for each complex structure $J\in {\mathcal C}^0_\omega(M)$, the unique potential with zero average. Hence, since $M$ is simply connected, there is a unique prequantum line bundle $({\mathcal L},\nabla, \langle\cdot,\cdot\rangle)$ with curvature $-i\omega$. Thus we satisfy all assumptions of Theorem \ref{thm:2} and so we get the following corollary.

\begin{corollary}\label{HCC}
For the coadjoint orbit $M$, we get a Hitchin connection in the bundle $H^{(k)}$ over the subspace ${\mathcal C}^0_\omega(M)$. This connection is invariant under the local action of the group of bundle automorphisms of the prequantum line bundle $({\mathcal L},\nabla, \langle\cdot,\cdot\rangle)$ covering the symplectomorphism group of $(M,\omega)$.
\end{corollary}

We see that this connection therefore satisfies all the requirements of Ginzburg and Montgomery's Theorem \ref{thm:1} above, thus this connection cannot be projectively flat over ${\mathcal C}^0_\omega(M)$. It still remains an interesting question to compute the curvature of this connection and to understand to what extend this connection fails to be projectively flat.

\section{Pullbacks of the Hitchin Connection}
\label{sec:hitch-conn-pullb}

Let us consider a symplectic manifold $(M, \omega)$, and assume that we have a rigid subfamily $\mathcal{T}\subseteq \mathcal{C}_\omega(M)$ of all the complex structures compatible with $\omega$. 

Furthermore we assume, that we have some connected subspace ${\mathcal C}_\omega^0(M)$, on which we can find a map $\Phi\colon \mathcal{C}^0_\omega(M) \to \Diff(M)$ denoted $J\mapsto \Phi_J$, such that for each $J$ there exists a $J'\in \mathcal {T}$ with
\begin{equation*}
  \Phi_J^*(J')=J \quad \text{and}\quad  \Phi_{\mid \mathcal{T}}=\Id.
\end{equation*}
That is $\Phi_J$ gives a biholomorphism from $M$ with the complex structure $J$ to $M$ with the complex structure $J'$ from the rigid family $\mathcal{T}$.

Now for  each $J$ we can consider the pullback bundle $\Phi_J^*\mathcal{L} \to M$, which is naturally isomorphic to $\mathcal{L}$ itself, since $\Phi_J$ is isotopic to the identity for all $J\in \mathcal{C}_\omega^0(M)$. 

Choosing a holomorphic isomorphism $\widetilde{\Psi}_J\colon \mathcal{L}\to \Phi_J^*\mathcal{L}$ we get the following commutative diagram.
  \begin{equation}
    \label{dia1}
    \xymatrix{
        \mathcal{L}\ar[d] \ar[r]^{\widetilde{\Psi}_J} & \ar[d] \ar[r]^p \Phi_J^*\mathcal{L} &\ar[d] \mathcal{L} \\
        M \ar[r]_{\Id} & M \ar[r]_{\Phi_J} & M
}
  \end{equation}
where $p$ is the map given canonically in the construction of the pullback bundle. Composing the maps in the top of the diagram, we get an induced endomorphism on $\mathcal{L}$ given by  $\Psi_J=p\circ\widetilde{\Psi}_J$. We need to fix $\Psi_J$ uniquely up to the action of the automorphism group of the line bundle, $\Aut(\mathcal{L})=C^*$.  We seek a section $\Psi$ of the bundle
\begin{equation*}
  \mathcal{L}(M)=\set{(J,\Psi)\in \mathcal{C}^0_\omega(M)\times \Hom(\mathcal{L},\mathcal{L})\mid  \Psi\colon (L,J)\to (L,J') \text{ holo. for some } J'\in \T}
\end{equation*}
over $\mathcal{D}(M)$, where
\begin{equation*}
  \mathcal{D}(M)=\set{(J,\Phi)\in \mathcal{C}^0_\omega(M)\times \Diff(M) \mid J=\Phi^*(J') \text{ for some } J'\in \T},
\end{equation*}
which in turn is a bundle over ${\mathcal C}^0_\omega(M)$.
If we have one point $x\in M$, which is fixed for all $\Phi_J$, $J\in {\mathcal C}^0_\omega(M)$, then we can fix the ambiguity by requiring that 
\begin{equation*}
  (\Psi_J)_{x}=\Id\colon \mathcal{L}_{x} \to \mathcal{L}_{x},
\end{equation*}
and hence get the required section $\Psi$.
Let us now assume we have a map $$\pi_\mathcal{T}\colon \mathcal{C}^0_\omega(M) \to \mathcal{T},$$ which is compatible with some $\Phi$, then any section $\Psi$ as above will induces an isomorphism of the bundles
\begin{equation}\label{diag2}
  \xymatrix{
    H^{(k)}_{|\mathcal{T}} \ar[dr] &\ar[l] \pi_{\mathcal{T}}^*((H^{(k)})_{|\mathcal{T}}) \ar[rd]  \ar[rr]^{\cong} && H^{(k)} \ar[ld]  \\
    & \mathcal{T} & \ar[l]_{\pi_\mathcal{T}}\mathcal{C}^0_\omega(M)&
  },
\end{equation}
and thus a projectively flat connection on $(H^{(k)})_{|\mathcal{T}}$, which we have by \cite{AG1} induces a projectively flat connection on the pullback bundle, which then gives a projectively flat connection on $H^{(k)}$. Here we have used that 
$$
  \pi^*((H^{(k)})_{|\mathcal{T}})_J=H^0(M_J, \Phi_J^*\mathcal{L}^k),$$
giving us the isomorphism on each fiber, and since the diagram \ref{diag2} commutes, an isomorphism on the level of bundles is obtained.

\begin{ex}
  One example where we can construct $\Phi$ fulfilling the requirements for $(M,\omega)$ is the underlying symplectic manifold of $\P^n$, where of course $\omega$ is the Fubini-Study symplectic form.

On $M$ every complex structure in a small enough neighbourhood $ {\mathcal C}^0_\omega(M)$ of the standard complex structure ($\P^n$ giving $J_0$) on $M$ is biholomorphic to the standard complex structure, so we let $\mathcal{T}=\set{J_0}$, and thus we have $\mathcal{D}(M)_J\neq \emptyset$ for all $J\in {\mathcal C}^0_\omega(M)$.

It is known that $\P^n$ has the property, that there exist $n+1$ points $x_0, x_1, \ldots, x_n\in \P^n$ such that the any set of lifts of these to ${\mathbb C}^{n+1}$ is a basis. Further, we have for any such set of $n+1$ points that there exists a unique $\Phi_0\in \Aut(\P^n)$ mapping $\Phi_0(y_i)=x_i$. This means that we can for any $J\in {\mathcal C}^0_\omega(M)$ determine a unique biholomorphism $\Phi_J : (M, J) \ra (M, J_0)$ such that $\Phi_J(x_i) = x_i$. This way, we can define a section $\Phi \in C^\infty(\mathcal{C}^0_\omega(M),D(M))$.

Note that $\Phi_{J_0}=\Id$, thus $\Phi$ and ${\mathcal T} = \{J_0\}$ fullfils the requirements outlined above. Thus we get a flat connection in $H^{(k)}$ over the entire space ${\mathcal C}^0_\omega(\P^n)$. This connection does however not have the symmetry required by Theorem \ref{thm:1}. Furthermore it does not agree with the connection obtained in Corollary \ref{HCC}, since that connection is not projectively flat, however that connection does have the symmetry properties. 

We expect that a similar construction can be made to work for any coadjoint orbit, by a similar "symmetry breaking" construction.

\end{ex}

\section{Further examples}
\label{sec:examples}
In this sections, we give a number of examples, where we can solve the weakly restricted criterion for open subsets of the entire family of complex structures on a given symplectic manifold and thus get a Hitchin connection on such subspace of all complex structures on the given symplectic manifold.

  The first example we consider is $M=\R{2n}$ with the standard symplectic structure $J_0$ and ${\mathcal C}^0_\omega(M)$ an open and small enough neighbourhood of $J_0$, such that $\mathcal{H}^{(k)}\mid_{{\mathcal C}^0_\omega(M)}$ is a vector bundle over ${\mathcal C}^0_\omega(M)$.

We may also assume that
 $$H^1(M_J, T'M_J) = 0$$
 for all complex structures $J\in {\mathcal C}^0_\omega(M)$. This means that we have a solution to the weakly restricted criterion with $G_\beta(V)=0$ and $\beta(V)_\sigma$ a solution to 
\begin{equation*}
  V'[J]_\sigma=  -\delbar_\sigma \beta(V)_\sigma
\end{equation*}
for all vector fields $V$ and points $\sigma$ on ${\mathcal C}^0_\omega(M)$. We will need a smooth family of $\beta$'s,which we can assume exists by choosing a suitable $ {\mathcal C}^0_\omega(M)$. 

In this case the functions $\phi(V)$ and $\psi(V)$ in the expression of the Hitchin connection from Theorem \ref{thm:2} can be calculated explicitly by a curve integral (depending of course on a choice of base point) of the $\delbar$-exact forms that they are related to by definition.

Now let us consider the symplectic torus $M=\R{2n}/\Z^{2n}$ with the standard symplectic structure $\omega$. In this case, it is not true that the moduli space of complex structures is locally a point. We consider the usual moduli space of linear complex structures compatible with the standard symplectic structure, which is  the moduli space of principal polarised abelian varieties. In fact, the space of all linear complex structures on $\R{2n}$ compatible with $\omega$ can be identification with 
$$ {\mathbb H} =\{Z\in M_{n,n}({\mathbb C})\mid Z= Z^t, \text{Im}(Z) >0\}.$$
The complex structure corresponding to $Z\in {\mathbb H}$ we denote $J_Z$.
It is easy to check that the map
\begin{equation*}
  \cdot \omega : H^0(M_{J_Z}, S^2(T'M_{J_Z})) \ra  H^1(M_{J_Z}, T'M_{J_Z})_\omega
\end{equation*}
is surjective for all $Z\in {\mathbb H}.$ Consider now the maximal connected subspace ${\mathcal C}^0_\omega(M)$ of all complex structures  on $M$, which is compatible with $\omega$ and for which there exist a unique $Z \in {\mathbb H}$ and a unique biholomophism 
$$\Phi_{J}\colon (M, J) \to (M,J_Z),$$
which induces the identity on $H^1(M,{\mathbb Z})$ and which preserves $0\in M$.
Then we of course also have that
\begin{equation} \label{omegaiso}
  \cdot \omega : H^0(M_{J}, S^2(T'M_{J})) \ra  H^1(M_{J}, T'M_{J})_\omega
\end{equation}
is surjective for all complex structures $J\in {\mathcal C}^0_\omega(M)$ and further gives us  a natural projection map 
\begin{equation*}
  \pi \colon \mathcal{C}^0_\omega(M) \to \bH.
\end{equation*}

We now fix a prequantum line bundle $({\mathcal L}, \nabla, \langle\cdot,\cdot\rangle)$ over $(M,\omega)$.
Consider then the bundle of quantum spaces $\tilde{H}^{(k)} \to \bH$, with its usual Hitchin connection (see e.g. \cite{H,A1}) and further the pullback
\begin{equation*}
  \pi^*\tilde{H}^{(k)} \to \mathcal{C}^0_\omega(M).
\end{equation*}
Since each $\Phi_J$ induces the identity on the first homology, we see that $\Phi_J^*{\mathcal L} \cong {\mathcal L}$ and as further $\Phi_J(0) =0,$ we can find a section $\Psi$ as discussed above
which induced an isomorphism of the quantum bundles
\begin{equation*}
  \Psi^* : \pi^*\tilde{H}^{(k)} \ra H^{(k)} .
\end{equation*}
We now pull back the Hitchin connection in $\tilde{H}^{(k)}$ to $\pi^*\tilde{H}^{(k)}$ and push it to $H^{(k)}$ by this isomorphism, to get a projectively flat connection. Again by the no-go Theorem \ref{thm:1}, this connection cannot be natural, which is also clear from its construction. 

We will now show that the constructions of this paper applies to provide a construction of a natural connection in $H^{(k)}$ in certain directions over a subspace of ${\mathcal C}^0_\omega(M)$. Since we have that (\ref{omegaiso}) is surjective, we see that the equation (\ref{Gb}) can be solved for all $J\in {\mathcal C}^0_\omega(M)$ and all tangent vectors $V\in T_J{\mathcal C}_\omega^0(M)$. For any choice of solution to this equation, we get a $\bar \partial$-closed form $\omega\cdot \beta_J(V)$. However, the map
$$ \omega \cdot : H^0(M_J, T'M_J) \ra H^{0,1}(M)$$
is an isomorphism for all $J\in{\mathcal C}_\omega^0(M)$, so we can uniquely determine $\beta_J(V)$ as a solution to (\ref{Gb}), by requiring that $\omega\cdot \beta_J(V) =0 $ in $H^{0,1}(M)$. This in turn means that we can indeed find a unique solution to the equation $\bar\partial \phi(V) = \omega\cdot \beta_J(V) $ of zero average. 

We now consider the linear map 
$$ [\Omega]_J : T_J{\mathcal C}_\omega^0(M) \ra H^{0,1}(M_J).$$
We see that we can apply our Hitchin connection construction along the distribution 
$$\ker [\Omega] \subset T{\mathcal C}_\omega^0(M),$$
simply by just choosing the $\phi(V)$ with zero average which solves
$$\bar \partial \phi(V) = \Omega(V).$$

 The last example we will consider is the moduli spaces $M$ of flat $SU(n)$-connections on a surface of genus $g>1$ possibly with central holonomy around a point on the surface. We have the Seshadri-Goldman-Atiyah-Bott symplectic form $\omega$ on (the smooth part of) $M$ \cite{AB}. We further have the Chern-Simons functional induces the Chern-Simons line bundle $({\mathcal L},\nabla, \langle\cdot,\cdot\rangle)$ over $(M,\omega)$ \cite{Fr,RSW}.
 
 We first consider the usual family of complex structures $J$ parametrized by Teichm\"{u}ller space $\T$. In this situation Hitchin has proved \cite{H} that the map
\begin{equation*}
  \cdot \omega : H^0(M_J, S^2(T'M_J)) \ra  H^1(M_J, T'M_J)_\omega
\end{equation*}
is surjective, for all $J\in \T$.  In analogy with the above torus case, we define ${\mathcal C}^0_\omega(M)$ to be the maximal connected subspace of all complex structures  on $M$, which is compatible with $\omega$ and for which there exist a unique $J' \in \T$ and a unique biholomophism 
$$\Phi_{J}\colon (M, J) \to (M,J'),$$
with the property that it varies smoothly with $J\in {\mathcal C}^0_\omega(M)$ and $\Phi_J = \id$ for all $J\in \T$. But now we see as above that (\ref{Gb}) can always be solved and since there are no holomorphic vector fields on $(M,J)$ for all $J\in {\mathcal C}^0_\omega(M)$ we get a unique $\beta_J(V)$ with the needed properties for all $J\in {\mathcal C}^0_\omega(M)$ and $V\in T_J{\mathcal C}^0_\omega(M)$ . But then we have that Theorem \ref{thm:2} applies and we we get a Hitchin connection in $H^{(k)}$ over all of ${\mathcal C}^0_\omega(M)$. We have here normalized $\psi(V)$ and $\phi(V)$ by requiring that they have zero average over $M$.

\end{document}